\documentclass[12pt]{amsart}
\usepackage{amssymb,amsmath,enumitem}
\usepackage[foot]{amsaddr}
\usepackage{graphicx, xcolor, soul}
\usepackage{soulpos} 
\soulregister{\cite}{7} 
\soulregister{\ref}{7}
\soulregister{\eqref}{7}

{\color{red}}%
{}
\usepackage{comment}
\newcommand{\R}{\mathbb{R}}

\newcommand{\M}{\mathcal{M}}

\newcommand{\G}{\mathcal{G}}

\newcommand{\be}{\begin{equation}}
\newcommand{\ee}{\end{equation}}
\newcommand{\bee}{\begin{equation*}}
\newcommand{\eee}{\end{equation*}}
\newcommand{\bea}{\begin{eqnarray}}
\newcommand{\eea}{\end{eqnarray}}
\newcommand{\bess}{\begin{eqnarray*}}
\newcommand{\eess}{\end{eqnarray*}}

\numberwithin{equation}{section}
\usepackage{hyperref}
\hypersetup{colorlinks=true, citecolor=teal, linkcolor=magenta, urlcolor=blue}
\theoremstyle{plain}
\newtheorem{Thm}{Theorem}[section]
\newtheorem{Cor}[Thm]{Corollary}
\newtheorem{Lem}[Thm]{Lemma}
\newtheorem{Prop}[Thm]{Proposition}

\theoremstyle{definition}
\newtheorem{Def}[Thm]{Definition}
\newtheorem{Rem}[Thm]{Remark}

\theoremstyle{remark}

\begin{document}
\title[Nonlocal Sublinear Elliptic Problems]
{Nonlocal Sublinear Elliptic Problems Involving Measures}
  
\author{Aye Chan May}
\address[Aye Chan May]{Sirindhorn International Institute of Technology, Thammasat University, Pathum Thani 12120, Thailand}
\email{\href{d6622300199@g.siit.tu.ac.th }{d6622300199@g.siit.tu.ac.th (A.C. May)}}

\author{Adisak Seesanea}
\address[Corresponding Author: Adisak Seesanea]{Sirindhorn International Institute of Technology, Thammasat University, Pathum Thani 12120, Thailand}
\email[A. Seesanea]{\href{adisak.see@siit.tu.ac.th}{adisak.see@siit.tu.ac.th (A. Seesanea)}}

\maketitle
\begin{abstract}
We study Dirichlet problems for fractional Laplace equations of the form
$(-\Delta)^{\frac{\alpha}{2}} u = f(x,u)$ in $\mathbb{R}^{n}$
for $0<\alpha<n$ where the nonlinearity $f(x,u) = \sum_{i=1}^{M} \sigma_{i} u^{q_i} + \omega$
involves sublinear terms with $0<q_{i}<1$ and the coefficients $\sigma_{i}, \omega$ are 
nonnegative locally finite Borel measures on $\mathbb{R}^n$.
We develop a potential theoretic approach for the existence of positive minimal solutions 
in Lorentz spaces to the problems under certain assumptions on $\sigma_{i}$ and $\omega$. 
The uniqueness properties of such solutions are discussed.
Our techniques are also applicable to similar sublinear problems
on uniform bounded domains when $0<\alpha< 2$, or on arbitrary domains with 
positive Green's functions in the classical case $\alpha =2$. 
\end{abstract}
\setcounter{tocdepth}{1}
\tableofcontents
\section{Introduction}\label{sect:intro} 
We consider Dirichlet problems of the type
\be \label{fractionalproblem}
\begin{cases}
	(-\Delta)^{\frac{\alpha}{2}} u=f(x,u),  \;\;\;  u \geq 0,  \;\; \quad\text{in}\quad \mathbb{R}^n, \\
	\liminf\limits_{x\to\infty} u(x) =0, 
	 \end{cases}
\ee
where $(-\Delta)^{\frac{\alpha}{2}}$ is the fractional Laplacian in $\R^{n}$ for $0<\alpha<n$ and $n \geq 2$.  The nonlinearity 
$f(x,u)=\sum\limits_{i=1}^{M} \sigma_i u^{q_i}+\omega$ is driven by several sublinear terms where $0<q_{i}<1$ and the coefficients $\sigma_i, \omega$ are nonnegative locally finite Borel measures on $\R^n$.  This class of measures will be denoted by $\M^{+}(\R^n)$.  

A solution $u \geq 0$ to \eqref{fractionalproblem} will be understood 
as a measurable function of the class $L^{q_i}_{loc}(\R^n, d\sigma_i)$ for all $i=1,2,\dots,M$ (so that each $u^{q_i} d\sigma_i \in \M^{+}(\R^n)$) which satisfies the corresponding integral equation
 \be \label{solofinteq}
u=\sum_{i=1}^{M}\mathcal{I}_\alpha(u^{q_i} d\sigma_{i})+\mathcal{I}_\alpha\omega \quad \text{in}\quad \R^n. 	
 \ee
If $u$ satisfies the inequality $u \geq \sum_{i=1}^{M}\mathcal{I}_\alpha(u^{q_i} d\sigma_{i})+\mathcal{I}_\alpha\omega$  in $\R^n$, in place of \eqref{solofinteq}, it is said to be a supersolution to \eqref{fractionalproblem}. We say that a nontrivial solution $u$ to \eqref{fractionalproblem} is minimal if $u \leq v$ in $\mathbb{R}^{n}$
for any nontrivial supersolution $v$ to \eqref{fractionalproblem}.

 Here $\mathcal{I}_{\alpha}\sigma$ stands for the Riesz potential of order $\alpha\in(0, n)$ of a measure $\sigma\in\mathcal{M}^+(\R^n)$, which is defined by
\[
\mathcal{I}_\alpha\sigma(x) =  
(-\Delta)^{-\frac{\alpha}{2}} \sigma(x) =
(n-\alpha)c(n,\alpha) \int_{0}^{\infty}\frac{\sigma(B(x,r))}{r^{n-\alpha}}\frac{dr}{r}, \quad x\in \mathbb{R}^n
\]
where $B(x,r) = \{ y \in \mathbb{R}^{n} : |x - y| < r \}$ is the ball centered at $x \in \mathbb{R}^{n}$ of radius $r > 0$.
Throughout, the normalization constant will be omitted for ease of notation.

Notice that $0< \mathcal{I}_{\alpha}\sigma \leq +\infty$ whenever $\sigma \not\equiv 0$. Moreover, 
$\mathcal{I}_{\alpha}\sigma < +\infty$ a.e. if and only if $\mathcal{I}_{\alpha}\sigma \not\equiv +\infty$. In this case, $\liminf\limits_{x\to\infty} \ \mathcal{I}_{\alpha}\sigma(x) =0$ \cite{Miz}. Therefore, the zero boundary condition in \eqref{fractionalproblem} is necessary as its solutions are understood in terms of Riesz potentials in \eqref{solofinteq}.

The main purpose of this paper is to investigate the existence and uniqueness problems for positive solutions in Lorentz spaces $L^{r,\rho}(\R^n)$, $0< r, \rho < \infty$, to Dirichlet problems involving several sublinear terms of the type \eqref{fractionalproblem}.

Let us observe that, in recent years, various classes of solutions to singular elliptic problems of the form \eqref{fractionalproblem} with a single sublinear term ($M=1$), involving elliptic operators, including
the $p$-Laplacian, or more generally, the $\mathcal{A}$-Laplacian in $\R^{n}$, the fractional Laplacian and the classical Laplacian in domains, were extensively studied. We would like to refer the reader to, for instance, \cite{BBC}, \cite{Bz},  \cite{BK}, \cite{BS}, \cite{CV2},  \cite{CV3},  \cite{CV}, \cite{PP}, \cite{QV}, \cite{SV2},  \cite{SV3}, \cite{V1}, \cite{Ve} and the literature cited there. 

This work is primarily motivated by \cite{HS} where Hara and one of the authors solved a similar existence problem for 
solutions in $L^{r,\rho}(\R^n)$ to quasilinear elliptic equations involving the $p$-Laplacian in $\R^n$ with several sub-natural growth terms, 
by employing the norm and pointwise estimates of the nonlinear Wolff potentials. To the extent of the authors' knowledge, their 
method does not seem well-adapted when dealing with nonlocal equations of the type \eqref{fractionalproblem}.
    
We also note that two-sided pointwise estimates of a general positive solution $u$ to \eqref{fractionalproblem} in the case $M=1$ were recently established in \cite{V3}:
\be \label{bilateral}
 c^{-1} \big[ \big(\mathcal{I}_\alpha\sigma_{1} \big)^{\frac{1}{1-q}}+\mathcal{K}_{\alpha}\sigma_{1}\big]+\mathcal{I}_\alpha\omega
 \leq u \leq  c \big[ \big(\mathcal{I}_\alpha\sigma_{1} \big)^{\frac{1}{1-q}}+\mathcal{K}_{\alpha}\sigma_{1}\big]+\mathcal{I}_\alpha\omega
\ee
 where $c = c(\alpha, q)$ is a positive constant. 
 
 The fractional {\it intrinsic} potential $\mathcal{K}_{\alpha}\sigma$
 of a measure $\sigma \in \mathcal{M}^{+}(\mathbb{R}^n)$, introduced in \cite{CV2}, is defined for 
 $0 < \alpha < n$ and $0<q<1$ by   
 \[
 \mathcal{K}_{\alpha}\sigma(x)=\int_{0}^{\infty} \frac{\varkappa(B(x,r))^\frac{q}{1-q}} {r^{n-\alpha}} \frac{dr}{r},\quad x\in \R^n
 \]
where $\varkappa(B(x,r))$ is the least positive constant in the localized weighted norm inequality,
for $\sigma_{B} = \sigma |_{B(x, r)}$, 
\[
\|\mathcal{I}_\alpha \nu \|_{L^q(\R^n, \,d\sigma |_{B(x, r)})}\leq \varkappa(B(x,r)) \, \nu (\R^n), \qquad \forall \nu \in \mathcal{M}^+(\R^n).
\]

In light of \eqref{bilateral}, a necessary and sufficient condition for the existence of a positive solution $u\in L^{r,\rho}(\mathbb{R}^n)$ to  \eqref{fractionalproblem} with $M=1$  is given by
\[
\mathcal{I}_\alpha \sigma_{1} \in L^{\frac{r}{1-q},\frac{\rho}{1-q}}(\R^n), 
\quad \mathcal{K}_{\alpha} \sigma_{1} \in L^{r,\rho}(\R^n)
\quad \text{and} \quad \mathcal{I}_\alpha\omega\in L^{r,\rho}(\R^n).
\]
Nevertheless, it is not quite easy to verify the second condition due to the way the intrinsic potential 
$\mathcal{K}_{\alpha} \sigma_{1}$ is defined.

We would like to point out that the uniqueness of solutions to nonlocal sublinear problems of the type \eqref{fractionalproblem} in the case $M=1$ has been proved in \cite{V4}. In addition, Phuc and Verbitsky also treated uniqueness for equations of $p$-Laplace type for general reachable solutions, in particular all solutions in various specific classes, again when $M=1$.
More recently, in \cite{MS2}, the authors established the existence of positive solutions in Lebesgue spaces to \eqref{fractionalproblem} provided $M=1$ and $\alpha = 2$, using generalized Green energy estimates. 

In this present study, we provide useful sufficient conditions on $\sigma_{i}$ and $\omega$ in terms of Riesz potentials 
for the existence of the positive minimal solution $u\in L^{r,\rho}(\R^n)$ to the Dirichlet problem
\eqref{fractionalproblem} in the case $M \geq 1$. 
We will see that these conditions can be used to deduce simple sufficient condition of the form: 
\begin{equation}\label{suffconfunction}
\sigma_i\in L^{s_i,t_i}_{+}(\R^n), \;\; \ i = 1,\dots,M
\qquad \text{and} \qquad
\omega\in L^{\tilde{s},\tilde{t}}_{+}(\R^n)
\end{equation}
with appropriate exponents $0 < s_{i}, t_{i}, \tilde{s}, \tilde{t} < \infty$, for the existence of such solution to \eqref{fractionalproblem}.
In particular, this yields a refinement of the existence result obtained in \cite{BO}.

Our main results can be stated as follows. 
\begin{Thm} \label{thm:fractional}
Let $\sigma_{1}, \dots, \sigma_{M}, \omega \in \mathcal{M^+}(\R^n)$ such that $(\sigma_1\dots,\sigma_M, \omega) \neq (0,\dots0, 0)$. Let $\{q_i\}_{i=1}^{M} \subset (0,1)$ and $0<\alpha<n$ with $n \geq 2$. Suppose
\be \label{rieszpotsigma}
\quad\qquad\qquad\qquad\mathcal{I}_\alpha\sigma_i\in L^{\frac{\gamma+q_i}{1-q_i}}(\R^n, d\sigma_i),\qquad i=1,2,\dots,M
\ee
and
\be \label{rieszpotmu}
\mathcal{I}_\alpha\omega \in L^{\gamma}(\R^n, d\omega) 
\ee
for some $0<\gamma<\infty.$ Then there exists a positive minimal solution $u \in L^{r,\rho}(\mathbb{R}^n)$ to \eqref{fractionalproblem} where $r=\frac{n(\gamma+1)}{n-\alpha}$ and $\rho=\gamma+1$.

In particular, \eqref{rieszpotsigma} and \eqref{rieszpotmu} are valid with $\gamma=1$ 
if and only if there exists a positive (minimal) finite energy solution $u \in \dot{H}^{\frac{\alpha}{2}}(\R^n)$ to \eqref{fractionalproblem}.
In this case, $u$ is the only solution in $\dot{H}^{\frac{\alpha}{2}}(\R^n)$ provided $0<\alpha\leq 2$. 
\end{Thm}
Here the fractional homogeneous Sobolev space $\dot{H}^{\frac{\alpha}{2}}(\R^n)$ is defined \cite{AH} in terms
of Riesz potential for $0<\alpha<n$ as
\[
\dot{H}^{\frac{\alpha}{2}}(\R^n)=\{u:u=\mathcal{I}_{\frac{\alpha}{2}}f,\; f\in L^2(\mathbb{R}^n)\}
\]
equipped with the norm $\|u\|_{\dot{H}^{\frac{\alpha}{2}}(\R^n) }=\|f\|_{L^2(\mathbb{R}^n)}$.
Notice that, according to the definitions above, a finite energy solution $u$ to \eqref{fractionalproblem} is
a nonnegative function that belongs to $L^{q_i}_{loc}(\R^n, d\sigma_i) \cap \dot{H}^{\frac{\alpha}{2}}(\R^n)$ for all $i=1,2,\dots,M$, so that $\int_{\R^{n}}|(-\Delta)^{\frac{\alpha}{4}}u|^{2}dx<+\infty$, and satisfies \eqref{solofinteq}, see also \cite[Definition 4.1]{SV1} and its remark for the case $M=1$.

A sufficient condition for \eqref{rieszpotsigma} and \eqref{rieszpotmu} is given by \eqref{suffconfunction}
with exponents:
\begin{equation}\label{exponents}
\begin{split}
 s_{i}=\frac{n(\gamma+1)}{n(1-q_i)+\alpha(\gamma+q_i)}, \quad  &t_{i}=\frac{\gamma+1}{1-q_i},  \quad \text{for} \;\; i=1,\dots,M,  \\[0.6em]
 \tilde{s}=\frac{n(\gamma+1)}{n+\alpha\gamma} \quad\; \text{and} \quad\; &\tilde{t}=\gamma+1.
\end{split}
\end{equation}
Thus, in view of Theorem \ref{thm:fractional}, we obtain the following consequence.
\begin{Cor}\label{cor:main}	
Under the assumptions of Theorem \ref{thm:fractional}, if \eqref{suffconfunction} with exponents \eqref{exponents} is fulfilled, then there exists a positive minimal solution $u\in L^{r,\rho}(\R^n)$ to \eqref{fractionalproblem}.
\end{Cor}

In our method of proof of Theorem \ref{thm:fractional}, we consider a general framework of corresponding integral equations of the form:
\be\label{solofinteqi}
u=\sum \limits_{i=1}^{M}\G( u^{q_i}d\sigma_{i})+\G\omega\quad\text{in}\quad\Omega
\ee
under some mild hypotheses on kernel $G$, including {\it quasi-symmetric} and {\it weak maximum principle} (see Section \ref{sec2}). These properties are satisfied by kernels associated with elliptic operators, for instance, the fractional Laplacian $(-\Delta)^{\frac{\alpha}{2}}$ in $\R^n$ with 
$0 < \alpha < n$, and the classical Laplacian $-\Delta$, or more generally, the linear uniformly elliptic operators with bounded measurable coefficients $-\text{div}(\mathcal{A} \,\nabla \cdot)$ in $\Omega$, see \cite{An}, \cite{GV}.

Here $\mathcal{G}\sigma$ stands for the potential of the measure $\sigma\in\mathcal{M}^+(\Omega)$ 
associated with positive lower semicontinuous $G$ on $\Omega \times \Omega$, which is defined by
\[
\mathcal{G}\sigma(x) =\int_\Omega G(x,y)d\sigma(y), \quad  x\in\Omega.
\]

To outline the ideas of our approach, we first establish, in Theorem \ref{thm:solofinteq}, necessary and sufficient conditions: 
\be\label{genkersigmai}
\G\sigma_i\in L^{\frac{\gamma+q_i}{1-q_i}}(\Omega,d\sigma_i), \quad i=1,2,\dots,M,
\ee
and
\be\label{genkerwithsigmai}
\G\omega \in L^{\gamma + q_{i}}(\Omega,d\sigma_{i}), \quad i=1,2,\dots,M,
\ee
 for the existence of a positive (minimal) solution $u$ that belongs to $L^{\gamma + q_{i}}(\Omega, d\sigma_{i})$ for all $i = 1,\dots,M$ to the integral equation \eqref{solofinteqi}. The construction of such a solution relies on the use of a characterization of \eqref{genkersigmai} in terms of 
weighted norm inequalities of the type \eqref{weightednormth} with $\sigma = \sigma_{i}$ and $q = q_{i}$ (see Theorem \ref{thm:weightednorm}).
The main new phenomena in this step are possible interactions, not only between $\sigma_{i}$ and $\omega$, but also between $\sigma_{i}$
and $\sigma_{j}$ for $i, j = 1,\dots, M$. These were properly handled by using a two-weight norm estimate in Lemma \ref{lemma2}.

Moreover, our techniques are  adapted to establish a sufficient condition for the existence of a positive minimal solution $u\in L^{r,\rho}(\Omega)$ to the problem
\be\label{classical}
\begin{cases}
-\Delta u=\sum\limits_{i=1}^{M}\sigma_i u^{q_i}+\omega\quad\text{in}\quad\Omega,\\
\;\liminf\limits_{x\to y} u(x)=0,\quad y\in \partial_\infty \Omega,
\end{cases}
\ee
where $\Omega\subset\mathbb{R}^n$ is an arbitrary domain (possibly unbounded) that has a positive Green function $g_\Omega(x,y)$. 

As above, a solution $u$ to \eqref{classical} is a nonnegative superharmonic function on $\Omega$
that belongs to $L^{q_i}_{loc}(\Omega, d\sigma_i)$ for all $i=1,2,\dots,M$, satisfying the integral 
equation \eqref{solofinteqi}, where $\mathcal{G}$ is the Green potential corresponding to $-\Delta$ on $\Omega$. When $``="$ in \eqref{solofinteqi} is replaced by $``\geq"$, $u$ is said to be a supersolution to 
\eqref{classical}. Minimality of a nontrivial solution $u$ to \eqref{classical} is understood in the sense
that $u \leq v$ in $\Omega$ for any nontrivial supersolution $v$ to \eqref{classical}. A solution $u$ to \eqref{classical} of the class $\dot{W}_0^{1,2}(\Omega)$, so that $\int_{\Omega} |\nabla u|^{2} dx < +\infty$, is called a finite energy solution to \eqref{classical}.

Here, the homogeneous Sobolev space $\dot{W}_0^{1,2}(\Omega)$ defined as the closure of $\mathcal{C}_0^{\infty}(\Omega)$ with respect to the (semi)norm $\|u\|_{\dot{W}_0^{1,2}(\Omega)}=\|\nabla u\|_{L^2(\Omega)}$, see \cite{KHM}, \cite{MZ}.

Our main results for this problem read as follows.
\begin{Thm} \label{thm:classical}
Let $(\sigma_1,\dots,\sigma_M,\omega) \in \mathcal{M^+}(\Omega)$  such that $(\sigma_1,\dots,\sigma_M,\omega) \neq (0,\dots,0,0).$ Let $0<q_i<1$ where $i=1,2,\dots,M$ and $g_\Omega$ be the positive Green function associated with $-\Delta$ in $\Omega\subset\mathbb{R}^n, n\geq 3$. Suppose also that \eqref{genkersigmai} and
\be\label{Greenpotmu} 
\mathcal{G}\omega \in L^{\gamma}(\Omega,d\omega)
\ee
with $0<\gamma<\infty$. Then there exists a positive minimal solution $u\in L^{r,\rho}(\Omega)$ to \eqref{classical}, where $r=\frac{n(\gamma+1)}{n-2}$ and $\rho=\gamma+1$.

In particular, there exists a positive (minimal) finite energy solution $u\in  \dot{W}_0^{1,2}(\Omega)$ to \eqref{classical} if and only if the two conditions \eqref{genkersigmai} and \eqref{Greenpotmu} hold with $\gamma=1.$ Moreover, such a solution is unique in $\dot{W}_0^{1,2}(\Omega)$.
\end{Thm}

A sufficient condition in Lorentz space for \eqref{genkersigmai} and \eqref{Greenpotmu} is given by 
\be \label{suffsigmaclassical}
\sigma_i\in L^{r_i,\rho_i}(\Omega),\quad r_i=\frac{n(\gamma+1)}{n(1-q_i)+2(\gamma+q_i)},\quad \rho_i=\frac{\gamma+1}{1-q_i}\quad
\ee
and
\be \label{suffmuclassical}
\omega\in L^{\tilde{r},\tilde{\rho}}(\Omega),\quad \tilde{r}=\frac{n(\gamma+1)}{n+2\gamma},\quad \tilde{\rho}=\gamma+1
\ee
The next corollary is a direct consequence of Theorem \ref{thm:classical} and can be deduced from Proposition \ref{prop}. 
\begin{Cor} \label{corollaryclassical}
Under the assumptions of Theorem \ref{thm:classical}, if condition \eqref{suffsigmaclassical} and \eqref{suffmuclassical} are fulfilled, then there exists a positive superharmonic solution $u\in L^{r,\rho}(\Omega)$ to \eqref{classical}. 
\end{Cor}
The improvement of Corollary \ref{corollaryclassical} has been achieved for the larger class of positive functions that cover the work of Boccardo and Orsina \cite{BO} in which the existence of solutions was done when $\sigma\in L^{r}(\Omega)$,  with suitable $r$ and $\Omega$ is a bounded open set.  

\subsection*{Organization of the paper}

The rest of the paper is organized as follows: In Section \ref{sec2},  we discuss some preliminary notions and basic concepts that will be used throughout the paper. Section \ref{sec3} is devoted to the proof of the key lemmas, the construction of a minimal positive solution to the general integral equation \eqref{solofinteqi}, and related potential estimates. As an application, we deduce the existence of 
positive minimal solutions in Lorentz spaces to the problems of the type  \eqref{fractionalproblem} and \eqref{classical}.
A study of the existence and uniqueness of finite energy solutions to these sublinear problems is illustrated in Section \ref{sec4}.

\section{Preliminaries}\label{sec2}
\subsection*{Function spaces}
Let $\Omega\subset\mathbb{R}^n$ be an arbitrary domain (possibly unbounded). We denote by $\mathcal{M}^+(\Omega)$ the set of all nonnegative locally finite Borel measures in $\Omega$ and by $C_0^\infty(\Omega)$ the set of all smooth compactly supported functions in $\Omega$.

For $\omega\in\mathcal{M}^+(\Omega)$ and $1\leq p<\infty$,  we denote by $L^p(\Omega,d\omega)$-space   of all real-valued mesurable functions $u$ on $\Omega$ such that 
\[
\|u\|_{L^p(\Omega,d\omega)}=\Big(\int_\Omega|u|^pd\omega\Big)^\frac{1}{p}< +\infty.
\]
We write $L^p(\Omega,dx)$ as $L^p(\Omega)$ simply when $dx=d\omega$ and denote Lebesgue measure.

Next we recall the definition of \textit{Lorentz space}. 
For $0<r,\rho\leq \infty$, we define
\[\|f\|_{L^{r,\rho}(\Omega)} :=
	\begin{cases}
	\Big(\int_{0}^{\infty}\big(t^{\frac{1}{r}}f^*(t)\big)^\rho \,\frac{dt}{t}\Big)^{\frac{1}{\rho}} &\text{if}\quad \rho<\infty;\\
	\sup_{t>0}t^\frac{1}{r}f^*(t) \ &\text{if}\quad \rho=\infty,
	\end{cases}
\]
	where $f^*$ is the decreasing rearrangement of $f$ which is defined by
	\[
	f^*(t)=\inf\{\alpha>0:|\{x\in\Omega:|f(x)|>\alpha\}|\leq t\}.
	\]
	The Lorentz space, denoted by $L^{r,\rho}(\Omega)$, is the set of all measurable functions $f$ on $\Omega$ such that $\|f\|_{L^{r,\rho}(\Omega)}<+\infty$.

Notice that Lebesgue spaces are Lorentz spaces, i.e., $L^{r,r} = L^r$ for $0<r\le \infty$. Moreover, for $0<p<q\le \infty$, $L^{r,p}$ is a subspace of $L^{r,q}$. 
More details of Lorentz spaces can be found in \cite{Gr}.
\subsection*{Kernels and potentials theory}

Let $G:\Omega\times\Omega\to(0,\infty]$ be a positive lower semicontinuous kernel. Every kernel in this paper will be considered to be this type, even if not stated explicitly. 

The potential   $\mathcal{G} \nu$ of a measure $\nu \in\mathcal{M}^+(\Omega)$ is given by 
\[
\mathcal{G} \nu(x)=\int_\Omega G(x,y) \,  d\nu(y), \quad  x\in\Omega.
\]

The kernel $G$ on $\Omega\times\Omega$ is said to satisfy the \textit{weak maximum principle} (WMP), with constant $h\geq 1$, if 
\[
\mathcal{G}\sigma(x)\leq 1,\quad \forall x\in\text{supp}(\sigma) \implies \mathcal{G}\sigma(x)\leq h,\quad \forall x\in\Omega,
\]
for any $\sigma\in\mathcal{M}^+(\Omega).$ Here we use the notation supp($\sigma$) for the support of $\sigma.$

When $h=1,$ we say that the kernel $G$ satisfies the \textit{strong maximum principle}.
\begin{Rem}
For every domain $\Omega\subset \mathbb{R}^n$ which possesses a Green function,  the strong maximum principle is fulfilled for the Green functions associated with the classical Laplacian $-\Delta$ and for the fractional Laplacian $(-\Delta)^{\frac{\alpha}{2}}$ in the case where $0<\alpha\leq 2.$ Moreover,  it  is also satisfied for Green kernel of uniformly elliptic operators in divergence form $\mathcal{L}=-\text{div}(\mathcal{A}\nabla \cdot)$ in place of $-\Delta.$ 
\end{Rem}
\begin{Rem}
The WMP (or boundedness principle) holds true  for Riesz kernels on $\mathbb{R}^n$ associated with $(-\Delta)^{\frac{\alpha}{2}}$ in the full range $0<\alpha <n,$ and more generally for all radially nonincreasing kernels on $\mathbb{R}^n$, see \cite{AH}.
\end{Rem}

A kernel $G:\Omega\times\Omega\to(0,+\infty]$ is \textit{quasi-symmetric} provided there exists a positive constant $a$ such that
\[
a^{-1} G(y,x)\leq G(x,y)\leq aG(y,x)\quad\forall x,y,\in\Omega.
\]
The quasi-symmetric condition is frequently used together with the WMP. Numerous kernels, including the Green kernel, Riesz kernel, and Na\"{i}m kernel, which are associated with elliptic operators, hold quasi-symmetry properties and satisfy the WMP. (see, e.g.,  \cite{An}, \cite{GV}).

In order to proceed further in the investigation of solutions, we need the extra restriction on the kernel $G$ studied here, including that $G$ is bounded above by the Riesz kernel $I_\alpha(x-y)=c|x-y|^{\alpha-n}$ of order  $\alpha$ for some $0<\alpha<n,$
\be\label{rieszkernel} 
G(x,y)\leq c I_\alpha(x-y)\quad\forall x,y\in\Omega,
 \ee
where $c$ is a positive constant independent of $x$ and $y.$

\begin{Rem}
It is well-known that $\|u\|_{ \dot{H}^{\frac{\alpha}{2}}(\Omega)}$ is equivalent to the Gagliardo seminorm
\[
\Big(\int_{\mathbb{R}^n}\int_{\mathbb{R}^n} \frac{|u(x)-u(y)|^2}{|x-y|^{n+\alpha}} \;dx\;dy \Big)^{\frac{1}{2}},
\]
when $0<\alpha<2$, see \cite{AH}.  
\end{Rem}
We will use the dual space of $\dot{H}^{\frac{\alpha}{2}}(\R^n)$  in the proof of Theorem \ref{thm:fractional}. The dual space is denoted by $\dot{H}^{-\frac{\alpha}{2}}(\mathbb{R}^n)$, consists of distributions $\sigma \in \mathcal{D}'(\mathbb{R}^n)$ such that 
\[
\|\sigma\|_{\dot{H}^{-\frac{\alpha}{2}}(\mathbb{R}^n)}=\sup \frac{|\langle \sigma,u \rangle|}{\|u\|_{\dot{H}^{\frac{\alpha}{2}}(\R^n)}}<+\infty,
\]
where the supremum is taken over all $u\in C_0^{\infty}(\mathbb{R}^n)$ and $0<\alpha<2$.  Thus, by duality, for a measure $\omega\in \mathcal{M}^+(\mathbb{R}^n)$, we have $ \sigma\in \dot{H}^{-\frac{\alpha}{2}}(\mathbb{R}^n)$ if and only if $\quad \|\mathcal{I}_{\frac{\alpha}{2}} \sigma\|_{L^2(\mathbb{R}^n)}<+\infty,$
  or equivalently, by $\int_{\mathbb{R}^n} \mathcal{I}_\alpha\sigma \;d\sigma<+\infty.$
  
We observe that $ \mathcal{I}_{\alpha}\sigma\in L^{\frac{1+q}{1-q}}(\mathbb{R}^n,d\sigma) $
 is equivalent to the trace inequality \cite{COV1}
 \be \label{riesz}
\|\mathcal{I}_{\frac{\alpha}{2}} g\|_{L^{1+q}(\mathbb{R}^n,d\sigma)}\leq c\|g\|_{L^2(\mathbb{R}^n)},\quad \text{for all}\;\; g  \in L^2(\mathbb{R}^n),
\ee
where $c$ is a positive constant independent of $g$.

The following theorem was established by Grigor'yan and Verbitsky \cite{GV} with more general form.  It gives the pointwise estimates for supersolutions to sublinear elliptic equations.

\begin{Thm}  [{\text See \cite{GV}}]   \label{thm:pointwise} 
Let $0<q<1$ and let $\sigma\in\mathcal{M^+}(\Omega)$ and let $G$ be a positive lower semicontinuous kernel on $\Omega\times \Omega$ that satisfies the WMP with constant $h\geq 1$. If $u\in L^q_{loc}(\Omega,d\sigma)$ is a positive supersolution  to the sublinear integral equation
	\be \label{subint}
	u\geq \mathcal{G}(u^qd\sigma),\quad x\in\Omega,
	\ee
	then
	\be \label{globlow}
	u(x)\geq (1-q)^{\frac{1}{1-q}}h^{-\frac{q}{1-q}}\big[\mathcal{G}\sigma(x)\big]^{\frac{1}{1-q}},\quad x\in\Omega.
	\ee 
\end{Thm}
We shall use the following iterated inequalities in more general which were established by Grigor'yan and Verbitsky \cite{GV}.
\begin{Thm}   [{\text See \cite{GV}}]    \label{thm:iterated} 
Let $\sigma\in\mathcal{M^+}(\Omega),$ and let $G$ be a positive lower semicontinuous kernel on $\Omega\times \Omega$ that satisfies the WMP with constant $h\geq 1$. Then the following estimates hold.\\
(i) If $a\geq 1,$ then
\be \label{iterated}
(\mathcal{G}\sigma)^a(x)\leq a h^{a-1}\,\mathcal{G}((\mathcal{G}\sigma)^{a-1}d\sigma)(x),\quad   x\in \Omega. 
\ee
(ii) If $0<a\leq1,$ then
\be \label{iteratedgeq}
(\mathcal{G}\sigma)^a(x)\geq a h^{a-1}\mathcal{G}((\mathcal{G}\sigma)^{a-1}d\sigma)(x),\quad x\in \Omega.
\ee
\end{Thm}
For the weighted norm inequality, the following result is due to Verbitsky \cite{V1}. 
\begin{Thm}  [{\text See \cite{V1}}]    \label{thm:weightednorm}  

Suppose that $\sigma\in\mathcal{M^+}(\Omega)$. Let $G$ be a positive lower semicontinuous kernel on $\Omega\times \Omega$ that satisfies the WMP.\\
(i) If $1<p<\infty$ and $0<r<p,$ then the weighted norm inequality 
\be \label{weightednorm}
\|\mathcal{G}(fd\sigma)\|_{L^r(\Omega,d\sigma)}\leq c\|f\|_{L^p(\Omega,d\sigma)},\quad f\in L^p(\Omega,d\sigma),
\ee
 is fulfilled if and only if
\be \label{greenpotsr}
\mathcal{G}\sigma \in L^{\frac{pr}{p-r}}(\Omega, d\sigma).
\ee
(ii) If $0<q<1$ and $0<\gamma<\infty,$ then there exists a positive (super)solution $u\in L^{\gamma +q}(\Omega,d\sigma)$ to sublinear integral equation \eqref{subint} if and only if the weighted norm inequality \eqref{weightednorm} is fulfilled with $r=\gamma+q$ and $p=\frac{\gamma+q}{q},$ i.e.,
\be \label{weightednormth}
\|\mathcal{G}(fd\sigma)\|_{L^{\gamma+q}(\Omega,d\sigma)}\leq c\|f\|_{L^{\frac{\gamma+q}{q}}(\Omega,d\sigma)},\quad  f\in L^{\frac{\gamma+q}{q}}(\Omega,d\sigma),
\ee
or equivalently,
\be \label{greenpotrq}
\mathcal{G}\sigma\in L^{\frac{\gamma+q}{1-q}}(\Omega,d\sigma).
\ee
\end{Thm} 
Quinn and Verbitsky \cite{QV} described the existence of solutions $u\in L^q(\Omega,d\sigma),$ and 
more generally $u\in L^q_{loc}(\Omega,d\sigma),$ by $(1,q)$-weighted norm inequalities, which 
corresponds to the case where $\gamma=0$ in \eqref{greenpotrq}.
In \cite{HM}, Havin and Maz'ya introduced the potential which is defined by 
\[
\mathcal{V}_{\frac{\alpha}{2},p}(f dx) (x)=\mathcal{I}_\frac{\alpha}{2}(\mathcal{I}_\frac{\alpha}{2}|f| dx)^{\frac{1}{p-1}}(x)\qquad x\in \mathbb{R}^n
\]
for $p\in(1,\infty), \alpha\in(0,n)$ and $f$ is a locally integrable function on $\mathbb{R}^n$. 
If $\alpha<n$, then $\mathcal{V}_{\frac{\alpha}{2},2}(f dx)=c\mathcal{I}_{\alpha}(|f| dx)$ for a suitable constant $c=c(\alpha,n)$. 
It means these potentials extend the Riesz potentials.                     
\begin{Thm}  [{\text See \cite{A1}}]  \label{thm:boundedofrieszPoten}
Let $p>1$, $n\geq2$, $0<\alpha<\frac{2n}{p}$ and $\Omega$ be a measurable subset of $\mathbb{R}^n$.  
If $0<t\leq\infty$ and $1<s<\frac{2n}{\alpha p},$ then there exists a constant $C=C(\alpha,p,n,s,t)$ such that
\be \label{boundedofriesz}
\|\mathcal{V}_{\frac{\alpha}{2},p}(f dx)\|_{L^{\frac{sn(p-1)}{n-\frac{s\alpha p}{2}},t(p-1)}(\Omega)}\leq C\|f\|^{\frac{1}{p-1}}_{L^{s,t}(\Omega)}
\ee
for every $f\in L^{s,t}(\Omega)$.	
\end{Thm}

\section{Solutions in Lorentz Spaces}\label{sec3}
In this section, we start with an important link  among coefficients  $\sigma$ and data $\omega$, which follows from conditions 
\be\label{Greenpotsigma}
\mathcal{G}\sigma\in L^{\frac{\gamma+q}{1-q}}(\Omega,d\sigma)
\ee
together with \eqref{Greenpotmu} yields an essential two-weight condition
  \be \label{Greenpotwithomega}
\mathcal{G}\omega\in L^{\gamma+q}(\Omega,d\sigma).
\ee 

The first lemma gives a sharp estimate for interaction between the coefficient $\sigma$ and the data $\omega$ appear in the sublinear problem of the form \eqref{fractionalproblem}. This lemma was stated in \cite[Lemma 4.3]{SV2}. However, the proof given there is valid only for sublinear growth case $0<q<1.$ Here, we give a simple proof in the extended to the range $-\gamma<q<1.$ 
\begin{Lem}\label{lemma2}
	Let $0<\gamma<\infty, -\gamma<q<1,$ and $\sigma,\omega \in \mathcal{M^+}(\Omega)$. Suppose 
$G$ is  a positive lower semicontinuous kernel on $\Omega\times\Omega$ that satisfies WMP. Then	
\be \label{eighth}
	\int_{\Omega}(\mathcal{G}\omega)^{\gamma+q}d\sigma \leq C  \Big[\int_{\Omega}(\mathcal{G}\omega)^{\gamma} d\omega \Big]^\frac{\gamma+q}{\gamma+1}\nonumber \times  \Big[\int_{\Omega}(\mathcal{G}\sigma)^{\frac{\gamma+q}{1-q}}d\sigma \Big]^\frac{1-q}{\gamma+1}
\ee
	where $C$ is a positive constant depending only on $\gamma$ and $q$.
\end{Lem}

\begin{proof}
Without loss of generality,  we may assume that $(\sigma, \omega) \not\equiv (0, 0)$ and both integrals on the right-hand side of the inequality are finite.  Consider the following two cases:
	\begin{itemize}
		\item Case $\gamma+q>1$.
	\end{itemize} 
	Applying the iterated inequality \eqref{iterated} with $a=\gamma+q$ together with Fubini's theorem and H\"{o}lder's inequality with the exponents $\frac{\gamma}{\gamma+q-1}$ and $\frac{\gamma}{1-q},$ we obtain
	\bea \label{first}
	\int_{\Omega}(\mathcal{G}\omega)^{\gamma+q}d\sigma&&\leq c\int_{\Omega}\mathcal{G}\big((\mathcal{G}\omega)^{\gamma+q-1}d\omega\big)\;d\sigma \nonumber \\		
	&&=c\int_{\Omega}(\mathcal{G}\omega)^{\gamma+q-1}\mathcal{G}\sigma \;d\omega
	\nonumber \\
	&&\leq c\Big[\int_{\Omega}(\mathcal{G}\omega)^\gamma d\omega \Big]^{\frac{\gamma+q-1}{\gamma}}\Big[\int_{\Omega}(\mathcal{G}\sigma)^{\frac{\gamma}{1-q}}\;d\omega\Big]^\frac{1-q}{\gamma}.
	\eea
	The second integral on the right-hand side of \eqref{first} is estimated by a similar argument as above. In fact, applying \eqref{iterated} again with $a=\frac{\gamma}{1-q},$ along with Fubini's theorem and H\"{o}lder inequality with the exponents $\frac{\gamma+q}{\gamma+q-1}$ and $\gamma+q,$ we deduce
	\bea \label{second}
	\int_{\Omega}(\mathcal{G}\sigma)^\frac{\gamma}{1-q}\;d\omega&&\leq c\int_{\Omega} \mathcal{G}\Big((\mathcal{G}\sigma)^{\frac{\gamma}{1-q}-1}d\sigma \Big)\;d\omega
	\nonumber \\
	&&=c\int_{\Omega}(\mathcal{G}\sigma)^\frac{\gamma+q-1}{1-q}\mathcal{G}\omega\;d\sigma \nonumber \\
	&&\leq c \Big[\int_{\Omega}(\mathcal{G}\sigma)^\frac{\gamma+q}{1-q}d\sigma \Big]^\frac{\gamma+q-1}{\gamma+q} \Big[\int_{\Omega}(\mathcal{G}\omega)^{\gamma+q}d\sigma \Big]^\frac{1}{\gamma+q}.
	\eea
	By \eqref{first} and \eqref{second}, we have
\[
\begin{split}
	\Big[\int_{\Omega}(\mathcal{G}\omega)^{\gamma+q}d\sigma \Big]^{1-\frac{1-q}{\gamma(\gamma+q)}}&\leq c  \Big[\int_{\Omega}(\mathcal{G}\omega)^\gamma \;d\omega \Big]^\frac{\gamma+q-1}{\gamma}\nonumber\\ &\times  \Big[\int_{\Omega}(\mathcal{G}\sigma)^\frac{\gamma+q}{1-q}\;d\sigma \Big]^\frac{(\gamma+q-1)(1-q)}{\gamma(\gamma+q)},
\end{split}
\]
which is finite by \eqref{Greenpotsigma} and \eqref{Greenpotmu}, and hence 
\eqref{Greenpotwithomega} holds. Letting  $C_1=1-\frac{1-q}{\gamma(\gamma+q)}.$ Multiplying 
exponents of both sides of previous estimate by $\frac{1}{C_1}$, we obtain the desired estimate.
	\begin{itemize}
		\item Case 2: $\gamma+q\leq1.$ 
	\end{itemize}
Write
	\[
	\int_{\Omega}(\mathcal{G}\omega)^{\gamma+q}d\sigma=\int_{\Omega}(\mathcal{G}\omega)^{\gamma+q}F^{t-1}F^{1-t}d\sigma,
	\]
	where $\frac{\gamma+q}{\gamma+1}<t<\gamma+q$ and $F$ is a positive $\sigma$-measurable function to be determined later. Applying H\"{o}lder's inequality with the exponents $\frac{1}{t}$ and $\frac{1}{1-t},$ we get
	\be \label{forth}
	\int_{\Omega}(\mathcal{G}\omega)^{\gamma+q}d\sigma\leq \Big(\int_{\Omega}(\mathcal{G}\omega)^{\frac{\gamma+q}{t}}F^{\frac{t-1}{t}}d\sigma \Big)^t \Big(\int_{\Omega} Fd\sigma\Big)^{1-t}.
	\ee
	Setting $F=(\mathcal{G}\sigma)^{\frac{\gamma+q}{1-q}}$ in \eqref{forth}, we obtain
	\bea \label{fifth}
	\int_{\Omega}(\mathcal{G}\omega)^{\gamma+q}d\sigma&&\leq\Big(\int_{\Omega}(\mathcal{G}\omega)^{\frac{\gamma+q}{t}}(\mathcal{G}\sigma)^{\big(\frac{\gamma+q}{1-q}\big)\big(\frac{t-1}{t}\big)}d\sigma\Big)^{t}\nonumber \\ &&\;\;\times \Big(\int_{\Omega}(\mathcal{G}\sigma)^{\frac{\gamma+q}{1-q}}d\sigma\Big)^{1-t}.
	\eea
	The first integral on the right-hand side of \eqref{fifth} is estimated by using iterated inequality \eqref{iterated} with $a=\frac{\gamma+q}{t},$ followed by Fubini's theorem, we get 

\[
\begin{split}
	\int_{\Omega}(\mathcal{G}\omega)^{\frac{\gamma+q}{t}}(\mathcal{G}\sigma)^{\frac{(\gamma+q)(t-1)}{(1-q)t}}d\sigma 
	&\leq c\int_{\Omega} \mathcal{G}\big((\mathcal{G}\omega)^{\frac{\gamma+q-t}{t}}d\omega\big) (\mathcal{G}\sigma)^{\frac{(\gamma+q)(t-1)}{(1-q)t}}d\sigma \\
&= \int_{\Omega} \big(\mathcal{G}\omega\big)^{\frac{\gamma+q-t}{t}}\big(\mathcal{G} \big(\mathcal
	{G}\sigma)^{\frac{(\gamma+q)(t-1)}{(1-q)t}}d\sigma\big)\big)d\omega.
\end{split}
\]

	We proceed by further estimating the second integral on the right-hand side of last display applying \eqref{iteratedgeq}  with $0<a=\frac{(\gamma+q)(t-1)}{(1-q)t}+1\leq 1$ and H\"{o}lder's inequality with the exponents $\frac{\gamma t}{\gamma+q-t}$ and $\frac{\gamma t}{a(\gamma+1)-(\gamma+q)},$ we deduce
	\begin{align*}
	\int_{\Omega}(\mathcal{G}\omega)^{\frac{\gamma+q}{t}}(\mathcal{G}\sigma)^{\frac{(\gamma+q)(t-1)}{(1-q)t}}d\sigma &\leq c \int_{\Omega} \big(\mathcal{G}\omega\big)^{\frac{\gamma+q-t}{t}}\big(\mathcal{G}\sigma\big)^{\frac{t(\gamma+1)-(\gamma+q)}{(1-q)t}}d\omega \\
	&\leq c \Big(\int_{\Omega}(\mathcal{G}\omega)^\gamma d\omega\Big)^{\frac{\gamma+q-t}{\gamma t}} \\ 
	&\quad\times\Big(\int_{\Omega}(\mathcal{G}\sigma)^{\frac{\gamma}{1-q}}d\omega\Big)^{\frac{t(\gamma+1)-(\gamma+q)}{\gamma t}}.
	\end{align*}
	Similarly, we deduce
	\begin{align*}
	\int_{\Omega}(\mathcal{G}\sigma)^{\frac{\gamma}{1-q}}d\omega &\leq \Big(\int_{\Omega}\mathcal{G}\sigma(\mathcal{G}\omega)^{\gamma+q-1}d\omega\Big)^ {\frac{\gamma}{1-q}}\Big(\int_{\Omega}(\mathcal{G}\omega)^{\gamma} d\omega\Big)^
	{\frac{1-\gamma-q}{1-q}}\\
	&=\Big(\int_{\Omega}\mathcal{G}\big((\mathcal{G}\omega)^{\gamma+q-1}d\omega\big)d\sigma\Big)^{\frac{\gamma}{1-q}}\Big(\int_{\Omega}(\mathcal{G}\omega)^\gamma d\omega\Big)^{\frac{1-\gamma-q}{1-q}}\\
	&\leq c \Big(\int_{\Omega}(\mathcal{G}\omega)^{\gamma+q}d\sigma\Big)^{\frac{\gamma}{1-q}}\Big(\int_{\Omega}(\mathcal{G}\omega)^{\gamma} d\omega\Big)^{\frac{1-\gamma-q}{1-q}}.
	\end{align*}
	Combining the preceding estimates and \eqref{fifth}, we have
	\[
	\Big(\int_{\Omega}(\mathcal{G}\omega)^{\gamma+q}d\sigma\Big)^{1-\frac{t(\gamma+1)-(\gamma+q)}{1-q}} \leq c\Big[\int_{\Omega}(\mathcal{G}\omega)^\gamma d\omega\Big]^{\frac{(\gamma+q)(1-t)}{1-q}}\Big[\int_{\Omega}(\mathcal{G}\sigma)^{\frac{\gamma+q}{1-q}}\Big]^{1-t}.
	\]
	This proves \eqref{Greenpotwithomega} holds, since both integrals on the right-side are finite by \eqref{Greenpotsigma} and \eqref{Greenpotmu}. Letting  $C_2=1-\frac{t(\gamma+1)-(\gamma+q)}{1-q}.$ By multiplying exponents of both sides of above equation by $\frac{1}{C_2},$ we obtain the required estimate.
	This completes the proof of lemma.
\end{proof}

The next theorem provides necessary and sufficient conditions for the existence of a positive solution (minimal) $u\in L^{\gamma+q_i}(\Omega,d\sigma_i)$ for each $i=1,2,\dots,M$ to the corresponding integral equation in the sublinear case. A complete proof for general kernels is given below.
\begin{Thm} \label{thm:solofinteq}  
Let $\{q_i\}_{i=1}^{M} \subseteq (0,1)$ and $0< \gamma < \infty.$ Let $(\sigma_1,\dots,\sigma_M,\omega)\in \mathcal{M^+}(\Omega)$ such that $(\sigma_1,\dots,\sigma_M,\omega) \neq (0,\dots,0,0)$. Suppose $G$ is a positive quasi-symmetric lower semicontinuous kernel on $\Omega \times \Omega$  satisfying the WMP. If  \eqref{genkersigmai} together with \eqref{genkerwithsigmai} is valid, then there exists a positive solution $u\in L^{\gamma+q_i}(\Omega,d\sigma_i)$ for all $i=1,2,\dots,M$ to 
\eqref{solofinteqi}.

Moreover,  $u$ is minimal in the sense that $u \leq v$ in $\Omega$ for any nontrivial supersolution $v$ to \eqref{solofinteqi}. 

Conversely, without quasi-symmetric assumption on $G,$ if there exists a positive supersolution $ u \in L^{\gamma+q_i}(\Omega,d\sigma_i)$ to \eqref{solofinteqi},  then
\eqref{genkersigmai} and \eqref{genkerwithsigmai} must be valid.
\end{Thm}

\begin{proof} 
	Suppose that \eqref{genkersigmai} and \eqref{genkerwithsigmai} are valid. For each $i=1,2,\dots,M,$ we define 
	\[
	u_i^{(0)}=\kappa (\G\sigma_i)^{\frac{1}{1-q_i}}
	\]
where $\kappa$ is a positive constant to be determined later. Then $u_i^{(0)}$ satisfies
\[
u_i^{(0)}\leq \kappa^{1-q}c\; \G\big( (u_i^{(0)})^{q_i} d\sigma_i\big)
\]
	where $q=\min\limits_{i=1,2,\dots,M} q_i$ and $c$ is the constant in \eqref{iterated}. Then 
	\be\label{baseu0}
	u_i^{(0)}\leq \G\big( (u_i^{(0)})^{q_i} d\sigma_i\big)
	\ee
provide the constant $\kappa$ small enough so that $\kappa\leq \frac{\kappa^q}{c},$ i.e. $\kappa\leq \big(\frac{1}{c}\big)^{\frac{1}{1-q}}$. Moreover, we can choose small constant $k=\min\{\big(\frac{1}{c}\big)^{\frac{1}{1-q}}, \tilde{C}\}$ where $\tilde{C}$ is the constant in Theorem \ref{thm:pointwise}, such that
\[
u_i^{(0)}\leq w
\]
whenever $w$ is a nonnegative supersolution to \eqref{solofinteqi}.
	Next we construct a nondecreasing sequence of positive functions $\{u_j\}^\infty_{j=0} \subseteq L^{\gamma+q_i}(\Omega,d\sigma_i),$ $ \forall  i=1,2,\dots,M$ as follows:
	\[
	u_0=\sum\limits_{i=1}^{M} \G\big ((u_i^{(0)})^{q_i} d\sigma_i \big)+\mathcal{G}\omega   
	\]
and
\[
	u_{j+1}=\sum\limits_{i=1}^{M} \G\big( u_j^{q_i} d\sigma_i \big)+\mathcal{G}\omega\quad   \text{for}\quad    j\in \mathbb{N}_0.
	\]
It clearly sees that $u_0$ is positive function as well since the two measures $(\sigma_1,\dots,\sigma_M,\omega)\neq(0,\dots,0,0)$.  Moreover, by \eqref{baseu0}, $u_i^{(0)}\leq u_0$. Consequently,
	\[
	u_{1}=\sum\limits_{i=1}^{M} \G\big( u_0^{q_i} d\sigma_i \big)+\mathcal{G}\omega\geq \sum\limits_{i=1}^{M} \G\big ((u_i^{(0)})^{q_i} d\sigma_i \big)+\mathcal{G}\omega=u_0.
		\]	
Suppose $u_0\leq u_1 \leq \dots \leq u_j$ for some $j\in \mathbb{N}_0.$ Then
	\[
	u_{j+1}=\sum\limits_{i=1}^{M} \G\big( u_j^{q_i} d\sigma_i \big)+\mathcal{G}\omega\geq \sum\limits_{i=1}^{M} \G\big( u_{j-1}^{q_i} d\sigma_i \big)+\mathcal{G}\omega=u_j.
	\]
	Therefore, by induction,  $\{u_j\}^\infty_{j=0}$ is defined as an increasing sequence of positive functions.
	
	We now claim that each $\{u_j\}_{j=1}^{M} \subseteq L^{\gamma+q_i}(\Omega,d\sigma_i)\;\forall i=1,2,\dots,M.$  
For each $i,$
\[
\|u_0\|_{L^{\gamma+q_i}(\Omega,\;d\sigma_i)}\leq \sum\limits_{l=1}^{M} \|\G\big( (u_l^{(0)})^{q_l} d\sigma_l \big)\|_{L^{\gamma+q_i}(\Omega,\;d\sigma_i)}+ \|\mathcal{G}\omega\|_{L^{\gamma+q_i}(\Omega,\;d\sigma_i)},
\]
where the second term on the right-hand side is finite by \eqref{genkerwithsigmai}. So,
\be \label{eq1}
\|u_0\|_{L^{\gamma+q_i}(\Omega,\;d\sigma_i)}\leq \sum\limits_{l=1}^{M} \|\G\big( (u_l^{(0)})^{q_l} d\sigma_l \big)\|_{L^{\gamma+q_i}(\Omega,\;d\sigma_i)}+C.
\ee
Now, by applying Lemma \ref{lemma2} with $\sigma=\sigma_i,$  $\omega=(u_l^{(0)})^{q_l} d\sigma_l$ and $q=q_i$, we find that
\be \label{eq2}
\begin{split}
 \|\G\big( (u_l^{(0)})^{q_l} d\sigma_l \big)\|_{L^{\gamma+q_i}(\Omega,\;d\sigma_i)} &\leq C \Big( \int_\Omega \G\big ( (u_l^{(0)})^{q_l}d\sigma_l \big) ^\gamma (u_l^{(0)})^{q_l}\;d\sigma_l   \Big)	^{\frac{1}{\gamma+1}}\\
&\quad\times\Big(\int_\Omega \big(\G\sigma_i\big)^{\frac{\gamma+q_i}{1-q_i}\; } d\sigma_i   \Big)^{\frac{1-q_i}{(\gamma+1)(\gamma+q_i)}}.
	\end{split}
	\ee
By  H\"{o}lder's inequality with the exponents $\frac{\gamma+q_l}{\gamma}$ and $\frac{\gamma+q_l}{q_l},$  we obtain
\be \label{eq3}
\begin{split}
\Big( \int_\Omega \G\big ( (u_l^{(0)})^{q_l}d\sigma_l \big) ^\gamma (u_l^{(0)})^{q_l}\;d\sigma_l   \Big)	^{\frac{1}{\gamma+1}} &\leq \|\G \big( (u_l^{(0)})^{q_l} d\sigma_l \big) \|_{L^{\gamma+q_l}(\Omega,\;d\sigma_l)}^{\frac{\gamma}{\gamma+1}}\\
&\quad\|u_l^{(0)}\|^{\frac{q_l}{\gamma+1}}_{L^{\gamma+q_l}(\Omega,\;d\sigma_l)}.
\end{split}
\ee
In view of \eqref{genkersigmai}, we have 
$
\mathcal{G}\sigma_{l}\in L^{\frac{\gamma+q_{l}}{1-q_{l}}}(\Omega,d\sigma_{l}),
$
which is equivalent (by Theorem \ref{thm:weightednorm}) to the weighted norm inequality \eqref{weightednormth} with $\sigma=\sigma_{l}$ and $q=q_{l}$:
\[
\|\G \big( f d\sigma_l \big) \|_{L^{\gamma+q_l}(\Omega,\;d\sigma_l)}\leq c\|f\|_{L^{\frac{\gamma+q_l}{q_{l}}}(\Omega,\;d\sigma_l)}, \quad  f\in L^{\frac{\gamma+q_{l}}{q_{l}}}(\Omega,d\sigma_{l}).
\]
Taking $f=(u_l^{(0)})^{q_l}\in L^{\frac{\gamma+q_l}{q_l}}(\Omega,d\sigma_l)$, we deduce
\be \label{eq4}
\|\G \big( (u_l^{(0)})^{q_l} d\sigma_l \big) \|_{L^{\gamma+q_l}(\Omega,\;d\sigma_l)}\leq c\|u_l^{(0)}\|^{q_l}_ {L^{\gamma+q_l}(\Omega,\;d\sigma_l)}.
\ee
By combining \eqref{eq2}, \eqref{eq3} and \eqref{eq4}, we obtain
\[
 \sum_{l=1}^{M}\|\G\big( (u_l^{(0)})^{q_l} d\sigma_l \big)\|_{L^{\gamma+q_i}(\Omega,\;d\sigma_i)}\leq Cc \sum_{l=1}^{M}\|u_l^{(0)}\|^{q_l}_{L^{\gamma+q_l}(\Omega,\;d\sigma_l)} \|\G\sigma_i\|_{L^{\frac{\gamma+q_i}{1-q_i}}(\Omega,d\sigma_i)}^{\frac{1}{\gamma+q_i}}
 \]
where the second term of the product on the right-hand side is finite by \eqref{genkersigmai}, therefore \eqref{eq1} becomes
 \[
 \|u_0\|_{L^{\gamma+q_i}(\Omega,\;d\sigma_i)}\leq Cc \sum_{l=1}^{M}\|u_l^{(0)}\|^{q_l}_{L^{\gamma+q_l}(\Omega,\;d\sigma_l)}+C.
 \]
 Hence, $u_0 \in {L^{\gamma+q_i}(\Omega,d\sigma_i)}$ for all  $i=1,2,\dots,M$.
Suppose $u_0,\dots,u_j \in L^{\gamma+q_i}(\Omega,d\sigma_i)$ for some $j\in\mathbb{N}.$ 
 Observe that
 \[
\| u_{j+1}\|_{L^{\gamma+q_i}(\Omega,\;d\sigma_i)}\leq\sum\limits_{l=1}^{M} \|\G\big( u_j^{q_l} d\sigma_l \big)\|_{L^{\gamma+q_i}(\Omega,\;d\sigma_i)}+\|\mathcal{G}\omega\|_{L^{\gamma+q_i}(\Omega,\;d\sigma_i)}.
 \]
 By using the same argument as above, we observe that
 \[
 \|u_{j+1}\|_{L^{\gamma+q_i}(\Omega,d\sigma_i)}\leq Cc \sum\limits_{l=1}^{M} \| u_{j+1} \|^{q_l}_{L^{\gamma+q_i}(\Omega,\;d\sigma_i)}+C_2
 \]
 \[
\quad\qquad \qquad\qquad \leq C_1M \|u_{j+1}\|^q_{L^{\gamma+q_i}(\Omega,\;d\sigma_i)}+C_2
 \]
 where $q=\max\limits_{i=1,2,\dots,M} q_i.$
We estimate the first term on the right-hand side of above equation using Young's inequality,
 \[
 \|u_{j+1}\|_{L^{\gamma+q_i}(\Omega,\;d\sigma_i)}\leq (1-q)(C_1M)^{\frac{1}{1-q}}+ q \| u_{j+1} \|^{q}_{L^{\gamma+q_i}(\Omega,\;d\sigma_i)}+C_2.
 \]
 Hence, 
 \[
 \|u_{j+1}\|_{L^{\gamma+q_i}(\Omega,\;d\sigma_i)}\leq (C_1M)^{\frac{1}{1-q}}+\frac{C_2}{1-q}.
 \]
This implies that $u_{j+1} \in {L^{\gamma+q_i}(\Omega,d\sigma_i)}$.
 We have shown that $\{u_j\}_{j=0}^{\infty}$ is an increasing sequence of positive functions in $L^{\gamma+q_i}(\Omega,d\sigma_i)$ for all $i=1,2,\dots,M.$ Finally, using  the monotone convergence
 theorem to the sequence $\{u_j\}^\infty_{j=0},$ the pointwise limit $$u:=\lim\limits_{j\to \infty}{u_j} \in 
 L^{\gamma+q_i}(\Omega,d\sigma_i)\quad\forall i=1,2,\dots,M $$ and satisfies \eqref{solofinteqi}.

We now claim that the solution  $u$ to \eqref{solofinteqi} constructed above is minimal.
Let $v$ be any nontrivial supersolution to \eqref{solofinteqi}. We observe that
\[
u_0=\sum\limits_{i=1}^{M} \G\big ((u_i^{(0)})^{q_i} d\sigma_i \big)+\mathcal{G}\omega   \leq \sum\limits_{i=1}^{M} \G\big (v^{q_i} d\sigma_i \big)+\mathcal{G}\omega  \leq v,
\]
due to the choice of $u_i^{(0)}$ defined above.
Assume $u_j \leq v$ for some $j\in \mathbb{N}.$ Then,
\[
u_{j+1}=\sum\limits_{i=1}^{M} \G\big( u_j^{q_i} d\sigma_i \big)+\mathcal{G}\omega \leq \sum\limits_{i=1}^{M} \G\big (v^{q_i} d\sigma_i \big)+\mathcal{G}\omega   =v.
\]
Arguing, by induction, we find that
\[
u_{j-1}\leq u_j\leq u\quad \text{for all}\quad j\in \mathbb{N}.
\]
Hence,  $u=\lim\limits_{j\to \infty} u_j\leq v.$ This proves the minimality of $u.$
	
	Conversely, if there exists a positive supersolution $u\in L^{\gamma+q_i}(\Omega,d\sigma_i)$ to \eqref{solofinteqi}, it is obvious that \eqref{genkerwithsigmai} is valid. Moreover, \eqref{genkersigmai} follows from the global pointwise lower bound \eqref{globlow}:
	\[
	u(x)\geq C_i[\mathcal{G}\sigma_i(x)]^\frac{q_i}{1-q_i},\quad    x\in \Omega,
	\] 
	which does not require quasi-symmetry of $G$ and $C_i$ is the constant in \eqref{globlow}.
	\end{proof}

The main key inequality for existence of solutions in Lorentz space $L^{r,\rho}(\R^n)$ is provided in the next lemma. 
\begin{Thm} \label{thm:muforlorentz}
	Let $0<\beta<\infty$ and $G$ be a positive lower semicontinuous kernel on $\Omega\times\Omega$ that satisfies the WMP and \eqref{rieszkernel} for some $0<\alpha<n$. Then
	\[
	\|\mathcal{G}\sigma\|_{L^{r,\rho}(\Omega)}\leq C\Big(\int_{\Omega}(\mathcal{G}\sigma)^\beta \,d\sigma\Big)^{\frac{1}{\beta+1}}
	\]
	for any $\sigma\in\mathcal{M^+}(\Omega)$, where $r=\frac{n(\beta+1)}{n-\alpha}, \rho=\beta+1$ and $C$ is a positive constant depending only on $\alpha,n,r$ and $\rho.$
\end{Thm}
\begin{proof}
Let $f\in L^{\frac{r}{r-\beta},\frac{\rho}{\rho-\beta}}(\Omega)$ be a nonnegative function.
Its zero extension to $\R^n$ will also be denoted by $f$. In view of \eqref{rieszkernel}, we have 
	\[
	\mathcal{G}(f dx) \leq c \mathcal{I}_{\alpha}(f dx) = c \mathcal{V}_{\frac{\alpha}{2},2}(f dx)
	\qquad \text{in} \quad \Omega.
	\]
Applying Theorem \ref{thm:boundedofrieszPoten} with $p:=2$, $s:=\frac{n(\beta+1)}{(n+\alpha\beta)} \in (1, \frac{n}{\alpha})$ and $t:=\beta+1>0,$ boundedness of the Havin-Maz'ya potential in Lorentz spaces, yields
\[
\| \mathcal{G}(f dx)\|_{L^{\frac{sn}{n-s\alpha},t}(\Omega)}\leq   
c \|\mathcal{V}_{\frac{\alpha}{2},2}(f dx)\|_{L^{\frac{sn}{n-s\alpha},t}(\Omega)}\leq c\|f\|_{L^{s,t}(\Omega)}.
\]
Notice that
\[
r = \frac{sn}{n-s\alpha}, \quad \rho = t, \quad s = \frac{r}{r-\beta} \quad \text{and} \quad t = \frac{\rho}{\rho - \beta}.
\]
The preceding inequality becomes
	\[
	\|\mathcal{G}(f\,dx)\|_{L^{r,\rho}(\Omega)}\leq c\|f\|_{L^{\frac{r}{r-\beta},\frac{\rho}{\rho-\beta}}(\Omega)}.
	\]
Applying H\"{o}lder's inequality, along with the norm estimate above, we deduce
	\begin{align*}
	\int_{\Omega}(\mathcal{G}(f\,dx))^\beta f\,dx &\leq \|(\mathcal{G}(f\,dx))^{\beta}\|_{L^{\frac{r}{\beta},\frac{\rho}{\beta}}(\Omega)}\|f\|_{L^{\frac{r}{r-\beta},\frac{\rho}{\rho-\beta}}(\Omega)}\\
	&=\|\mathcal{G}(f\,dx)\|^\beta_{L^{r,\rho}(\Omega)}\|f\|_{L^{\frac{r}{r-\beta},\frac{\rho}{\rho-\beta}}(\Omega)}\\
	&\leq c\|f\|^{\beta}_{L^{\frac{r}{r-\beta},\frac{\rho}{\rho-\beta}}(\Omega)}\|f\|_{L^{\frac{r}{r-\beta},\frac{\rho}{\rho-\beta}}(\Omega)}\\
	&=c\|f\|^{\beta+1}_{L^{\frac{r}{r-\beta},\frac{\rho}{\rho-\beta}}(\Omega)}.
	\end{align*}
Given $\sigma\in\mathcal{M^+}(\Omega)$. Appealing to Lemma \ref{lemma2} with $d\omega:= d\sigma$, $d\sigma:=fdx$, $q:=0$ and $\gamma = \beta$, we estimate
	 \begin{equation}\label{ninth}
	 \begin{split}
	 \int_{\Omega}(\mathcal{G}\sigma)^\beta f\,dx 
	  &\leq C\Big[\int_{\Omega}(\mathcal{G}\sigma)^\beta \,d\sigma\Big]^{\frac{\beta}{\beta+1}} \times 
	  \Big[ \int_{\Omega}(\mathcal{G}(f\,dx))^\beta f\,dx \Big]^{\frac{1}{\beta+1}} \\
	 &\leq C\Big[\int_{\Omega}(\mathcal{G}\sigma)^\beta \,d\sigma\Big]^{\frac{\beta}{\beta+1}}\|f\|_{L^{\frac{r}{r-\beta},\frac{\rho}{\rho-\beta}}(\Omega)}.
	 \end{split}
	 \end{equation}
	Thus, by a dual characterization of $L^{\frac{r}{\beta},\frac{\rho}{\beta}}(\Omega)$ and \eqref{ninth},
	\begin{align*}
	\|\mathcal{G}\sigma\|^\beta_{L^{r,\rho}(\Omega)}&=\|(\mathcal{G}\sigma)^\beta\|_{L^{\frac{r}{\beta},\frac{\rho}{\beta}}(\Omega)}\\
	& \leq C \sup\big\{\int_{\Omega}(\mathcal{G}\sigma)^\beta f\,dx :\|f\|_{L^{\frac{r}{r-\beta},\frac{\rho}{\rho-\beta}}(\Omega)}\leq 1, f\geq 0\big\}\\
	&\leq  C\Big(\int_{\Omega}(\mathcal{G}\sigma)^\beta  \,d\sigma\Big)^{\frac{\beta}{\beta+1}}. \qedhere
	\end{align*}
\end{proof}
Our method of proof of Theorem \ref{thm:fractional} is also based on the weighted norm inequality with Lorentz norm on the left-side:
\be \label{weightednormLorentz}
\|\mathcal{G}(f d\sigma)\|_{L^{r,\rho}(\Omega)} \leq c \|f\|_{L^s(\Omega,\,d\sigma)},\; f\in L^s(\Omega,\,d\sigma)
\ee
where $c$ is a positive constant independent of $f$. The following theorem gives sufficient conditions for the validity of \eqref{weightednormLorentz} at the Lorentz level.
\begin{Thm} \label{thm:sigmaforlorentz} 
Let $0<q<1$ and let $\sigma \in \M^{+}(\Omega)$. 
Let $G$ be a positive lower semicontinuous kernel on $\Omega \times \Omega$ that satisfies 
the WMP and \eqref{rieszkernel} for some $0<\alpha< n$. 
Suppose that $\frac{n}{n-\alpha}<r<\infty$ and 
\be \label{neccondinlorentz}
\G\sigma\in L^{\frac{r}{1-q},\frac{\rho}{1-q}}(\Omega)
\ee
 holds. Then 
\be \label{weightednormwithsigma}
\| \mathcal{G}(f\,d\sigma) \|_{L^{r, \rho}(\Omega)} \leq c \| \mathcal{G}\sigma \|_{L^{\frac{r}{1-q}, \frac{\rho}{1-q}}(\Omega)}^{\frac{1}{s'}} \| f \|_{L^{s}(\Omega, \,d\sigma)},\quad f\in L^{s}(\Omega,d\sigma),
\ee
where $s = \frac{r(n-\alpha) - n(1-q)}{nq}$ and $c$ is a positive constant independent of $f$ and $\sigma$. 

\end{Thm}

\begin{proof}
Observe that the hypotheses ensure 
\[
1<s<\infty,\quad \frac{(s'-1)r}{q}>1\quad \text{and}\quad \frac{1}{r'}-\frac{q}{(s'-1)r}=\frac{\alpha}{n}
\]
where $s'=\frac{s}{s-1}$ and $t'=\frac{t}{t-1}$.
By duality, \eqref{weightednormwithsigma} is equivalent to
\be \label{dualityofweightednorm}
\| \mathcal{G}(g \,dx) \|_{L^{s'}(\Omega,\,d\sigma)} \leq c \| \mathcal{G}\sigma \|_{L^{\frac{r}{1-q}, \frac{\rho}{1-q}}(\Omega)}^{\frac{1}{s'}} \| g \|_{L^{r',\rho'}(\Omega)},
\ee
where $c$ is a positive constant independent of $g$ and $\sigma.$ Without loss of generality, assume that $g\in L^{r',\rho'}{(\Omega)}$ with $g\geq 0.$ Applying the iterated inequality \eqref{iterated} with $t=s'$ followed by Fubini's theorem, we deduce that
\begin{align*}
\int_\Omega \Big(\mathcal{G}(g\,dx)\Big)^{s'}\,d\sigma &\leq c \int_{\Omega}\mathcal{G}\Big[\Big(\mathcal{G}(g\,dx)\Big)^{s'-1} g\,dx\Big]\,d\sigma \\
&=c\int_{\Omega}\Big(\mathcal{G}(g\,dx)\Big)^{s'-1}(\mathcal{G}\sigma)g\,dx.
\end{align*}
Using H\"{o}lder's inequality with the exponents\ \ $\frac{r}{q},\frac{r}{1-q},r'\ \ \text{and}\ \  \frac{\rho}{q},\frac{\rho}{1-q},\rho',$ we estimate
\begin{align*}
\int_{\Omega}\Big(\mathcal{G}(g\,dx)\Big)^{s'-1}(\mathcal{G}\sigma) g\,dx &\leq\|\mathcal{G}(g\,dx)\|^{s'-1}_{L^{\frac{(s'-1)r}{q},\frac{(s'-1)\rho}{q}}(\Omega)}\\&\quad\times\|\mathcal{G}\sigma\|_{L^{\frac{r}{1-q},\frac{\rho}{1-q}}(\Omega)}\| g\|_{L^{r',\rho'}(\Omega)}.
\end{align*}
By the assumption \eqref{neccondinlorentz} and Theorem \ref{thm:boundedofrieszPoten}, we obtain 
\[
\|\mathcal{G}(g\,dx)\|_{L^{\frac{(s'-1)r}{q},\frac{(s'-1)\rho}{q}}(\Omega)}\leq c\|g\|_{L^{r',\rho'}(\Omega)}.
\]
Combining the preceding inequalities yields \eqref{dualityofweightednorm} as desired.
\end{proof}

The next corollary is a particular case of Theorem  \ref{thm:muforlorentz} in order to obtain the necessary condition for existence of solution.
\begin{Cor}\label{Cor}
Let $0<q<1$ and let $\sigma \in \M^{+}(\Omega)$. 
Let $G$ be a positive lower semicontinuous kernel on $\Omega \times \Omega$ that satisfies 
the WMP and \eqref{rieszkernel} for some $0<\alpha< n$. 
Then, for $0<\gamma<\infty,$ \eqref{Greenpotsigma} implies \eqref{neccondinlorentz} with $r=\frac{n(\gamma+1)}{n-\alpha}$ and $\rho=\gamma+1.$ 
\end{Cor}
\begin{proof}
	This follows from Theorem \ref{thm:muforlorentz} with 
\[
\beta = \frac{\gamma + q}{1-q},
\quad
r := \frac{r}{1-q}
\quad \text{and} \quad
\rho := \frac{\rho}{1-q}. \qedhere
\]
\end{proof}
Now,  we present a proof of our main result in Theorem \ref{thm:fractional}. 
The proof for the finite energy solutions part will be presented in Section \ref{sec4}.
\begin{proof}[Proof of Theorem \ref{thm:fractional}] 
Suppose that \eqref{rieszpotsigma} and \eqref{rieszpotmu} hold. By Lemma \ref{lemma2}, the condition 
\[
\mathcal{I}_\alpha\omega \in L^{\gamma+q_i}(\R^n, d\sigma_i),\quad i=1,2,\dots,M
\]
 is fulfilled. Thus, by Theorem \ref{thm:solofinteq}, there exists a positive minimal solution $u\in L^{\gamma+q_i}(\mathbb{R}^n,d\sigma_i)$ for $i=1,2,\dots,M$ to the integral equation
\[
u=\sum_{i=1}^{M} \mathcal{I}_\alpha(u^{q_i} d\sigma_i)+\mathcal{I}_\alpha\omega\quad\text{in}\quad\R^n.
\]
Now, in view of Theorem \ref{thm:muforlorentz} with $\beta=\gamma,$ Theorem \ref{thm:sigmaforlorentz} and Corollary \ref{Cor} (with $d\sigma=d\sigma_i$ and $q=q_i$), taking into account the two assumptions \eqref{rieszpotsigma} and \eqref{rieszpotmu}, we have 
\begin{align*}
\|u\|_{L^{r,\rho}(\R^{n})}
	&\leq \sum_{i=1}^{M} \|\mathcal{I}_\alpha(u^{q_i}d\sigma_i)\|_{L^{r,\rho}(\R^n)} + \|\mathcal{I}_\alpha \omega\|_{L^{r,\rho}(\mathbb{R}^n)}\\
	&\leq  c\sum_{i=1}^{M}\Big(\int\limits_{\R^n} \big(\mathcal{I}_\alpha\sigma_i\big)^{\frac{\gamma+q_i}{1-q_i}}\;d\sigma_i \Big)^{\frac{\gamma+q_i}{1-q_i}}   \|u^{q_i}\|_{L^{\frac{\gamma+q_i}{q_i}}(\R^n,\,d\sigma_i)} 
	\\&\quad+ C\Big(\int_{\R^n}(\mathcal{I}_\alpha\omega)^\gamma \,d\omega\Big)^{\frac{1}{1+\gamma}}\\
	&= c\sum_{i=1}^{M}\|u\|^{q_i}_{L^{\gamma+q_i}(\R^n,\,d\sigma_i)}   <+ \infty.
	\end{align*}
This shows that $u$ is the minimal positive solution of the class in $L^{r,\rho}(\R^n)$  to \eqref{fractionalproblem}.
\end{proof}

\begin{proof}[Proof of Theorem \ref{thm:classical}] 
Suppose that \eqref{genkersigmai} and \eqref{Greenpotmu} are valid.
 Lemma \ref{lemma2}, Theorem \ref{thm:solofinteq}, Theorem \ref{thm:muforlorentz}, and Theorem \ref{thm:sigmaforlorentz} can be performed in terms of Green potentials. Hence, combining these all lemmas gives the existence of a minimal solution $u \in L^{r,\rho}(\R^n)$ to \eqref{classical}. This completes the proof of Theorem \ref{thm:classical}.
\end{proof}

We shall need the next proposition to prove the Corollary \ref{cor:main}.
\begin{Prop}\label{prop}
Let $G$ be a positive lower semicontinuous kernel on $\Omega\times\Omega$ satisfying \eqref{rieszkernel} for some $0<\alpha<n.$ If $\omega\in L^{s,t}(\Omega)$ is a positive function, where
\[
	s= \frac{n(\beta + 1)}{n+\alpha\beta} 
	\qquad \text{and} \qquad 
	t= \beta + 1,
	\]
then 
\[
\G\omega\in L^{\beta}(\Omega,d\omega).
\]
\end{Prop}
\begin{proof}
	 Fix $\beta > 0$. Suppose $\omega\in L^{s,t}(\Omega)$ is a positive function.
		Observe that $s=\frac{n(\beta+1)}{n+\alpha \beta}>\frac{n(\beta+1)}{n+n\beta}=1$ and $t>1.$ Then
	\be \label{greenomegaholder}
	\int_{\Omega}(\mathcal{G}\omega)^\beta d\omega\leq \|\mathcal{G}\omega\|_{L^{\beta s',\beta t'}(\Omega)}\|\omega\|_{L^{s,t}(\Omega)},
	\ee
	where $s'$ and $t'$ are H\"{o}lder conjugate of $s$ and $t.$ Denote $\tilde{\omega}$ the zero extension of $\omega$ to $\mathbb{R}^n.$ By the assumption \eqref{rieszkernel} and the boundedness of Riesz potential, there is a positive constant $c\leq C$ such that 
	\bea \label{greenomegainlorentz}
	\|\mathcal{G}\omega\|_{L^{\beta s',\beta t'}(\Omega)}&\leq& C\|\mathcal{I}_{\alpha}\tilde{\omega}\|_{L^{\beta s',\beta t'}(\mathbb{R}^n)} \nonumber \\ &\leq& C\|\tilde{\omega}\|_{L^{s,t}(\mathbb{R}^n)}\nonumber \\ &=&C\|\omega\|_{L^{s,t}(\Omega)}.
	\eea
	Combining \eqref{greenomegaholder} and \eqref{greenomegainlorentz} yields $$\mathcal{G}\omega \in L^{\beta}(\Omega, d\omega).$$
\end{proof}
	We finish this section by providing a proof of Corollary \ref{cor:main}.
\begin{proof}[Proof of Corollary \ref{cor:main}] 

	Proposition \ref{prop} shows that \eqref{suffconfunction} with exponents \eqref{exponents}  imply \eqref{rieszpotsigma} and \eqref{rieszpotmu}. Consequently, Corollary \ref{cor:main} follows immediately from Theorem \ref{thm:fractional}.
\end{proof}
\begin{Rem}
Corollary \ref{corollaryclassical} can be proved the same argument as the proof of Corollary \ref{cor:main} with $\alpha=2$.
\end{Rem}
\section{Finite Energy Solutions} \label{sec4}
In this section, we establish the existence and uniqueness results for finite energy solutions to sublinear elliptic problems \eqref{fractionalproblem} and \eqref{classical}, stated as a part of Theorem \ref{thm:fractional} and Theorem \ref{thm:classical}, respectively. 
We refer to \cite{SV1} in the case of a single sublinear term $M=1$.

\begin{proof}[Proof of Theorem \ref{thm:fractional} (Finite Energy Solutions Part)]
Suppose that conditions \eqref{rieszpotsigma} and \eqref{rieszpotmu} hold with $\gamma=1$.
In light of Lemma \ref{lemma2}, this implies \eqref{genkerwithsigmai} with $\gamma=1$. 
By Theorem \ref{thm:solofinteq}, there exists a positive minimal solution $u$ to \eqref{fractionalproblem} belonging to $ L^{1+q_{i}}(\R^{n},d\sigma_{i})$ for all $i=1,2,\cdots,M$.
In order to prove $u\in \dot{H}^{\frac{\alpha}{2}}(\R^n),$ by duality, it is sufficient to  show that there exists a positive constant $C$ such that
\be \label{duality}
\Big| \int_{\mathbb{R}^n} u\,\psi dx \Big| \leq C\|\psi\|_{\dot{H}^{-\frac{\alpha}{2}}(\mathbb{R}^n)},\quad \psi\in C^\infty_0(\mathbb{R}^n).
\ee
By the semigroup property of the Riesz potentials, Tonelli's Theorem, and H\"{o}lder's inequality, we have

\begin{align*}\label{normwithpsi}
\Big| \int_{\mathbb{R}^n} u \,\psi dx \Big| & \leq \sum_{i=1}^{M} \int_{\mathbb{R}^n} \mathcal{I}_{\frac{\alpha}{2}}(u^{q_i} d\sigma_i) |\mathcal{I}_{\frac{\alpha}{2}}\psi | \,dx + \int_{\mathbb{R}^n} \mathcal{I}_{\frac{\alpha}{2}}\omega\, |\mathcal{I}_{\frac{\alpha}{2}}\psi |\, dx \nonumber
\\
\quad
&\leq \sum_{i=1}^{M}\|\mathcal{I}_{\frac{\alpha}{2}}(u^{q_i} d\sigma_i)\|_{L^2(\mathbb{R}^n)}\,\|\mathcal{I}_{\frac{\alpha}{2}}\psi\|_{L^2(\mathbb{R}^n)} \nonumber
\\
&\quad
+ \|\mathcal{I}_{\frac{\alpha}{2}}\omega\|_{L^2(\mathbb{R}^n)}      \|\mathcal{I}_{\frac{\alpha}{2}}\psi \|_{L^2(\mathbb{R}^n)}\nonumber
\\
&
=\Big[\sum_{i=1}^{M}\|\mathcal{I}_{\frac{\alpha}{2}}(u^{q_i} d\sigma_i)\|_{L^2(\mathbb{R}^n)} +\|\omega\|_{\dot{H}^{-\frac{\alpha}{2}}(\mathbb{R}^n)}\Big] \|\psi\|_{\dot{H}^{-\frac{\alpha}{2}}(\mathbb{R}^n)}
\end{align*}
for all $\psi\in C^\infty_0(\mathbb{R}^n)$. 
Recall that by Wolff's inequality \cite[Section 4.5]{AH} and duality, we have
\[
 \mathcal{I}_{\alpha}\omega \in L^{1}(\mathbb{R}^n, d\omega)
\quad \Leftrightarrow \quad 
\mathcal{I}_{\frac{\alpha}{2}}\omega \in L^{2}(\mathbb{R}^n)
\quad \Leftrightarrow \quad 
\omega \in \dot{H}^{-\frac{\alpha}{2}}(\mathbb{R}^n). 
\]
In particular, condition \eqref{rieszpotmu} with $\gamma=1$ yields $\|\omega\|_{\dot{H}^{-\frac{\alpha}{2}}(\mathbb{R}^n)}<+\infty$. Therefore, in view of \eqref{duality} and the preceding estimate, it remains to show that
\be \label{rieszL2}
\sum_{i=1}^{M}\|\mathcal{I}_{\frac{\alpha}{2}}(u^{q_i} d\sigma_i)\|_{L^2(\mathbb{R}^n)}<+\infty.
\ee
To this end, we note that it was shown that condition \eqref{rieszpotsigma} with $\gamma =1$ holds if and only if the trace inequality \eqref{riesz} is valid with $d\sigma:=d\sigma_{i}$ and 
 $q:=q_{i}$:
\[
\|\mathcal{I}_{\frac{\alpha}{2}} g\|_{L^{1+q_{i}}(\mathbb{R}^n,d\sigma_{i})}\leq c_{i}\|g\|_{L^2(\mathbb{R}^n)}\quad \text{for all}\;\; g  \in L^2(\mathbb{R}^n),
\]
where $c_{i}$ is a positive constant independent of $g$, or equivalently by duality, 
\be\label{rieszdualitynormwithgamma1}
\|\mathcal{I}_{\frac{\alpha}{2}} (\psi d\sigma_{i})\|_{L^2(\mathbb{R}^n)}\leq C_{i}\|\psi\|_{L^\frac{1+q_{i}}{q_{i}}(\mathbb{R}^n,d\sigma_{i})} \quad \text{for all}\;\; \psi\in L^{\frac{1+q_{i}}{q_{i}}}(\mathbb{R}^n,d\sigma_{i}),
\ee
where $C_{i}$ is a positive constant independent of $\psi$, see  \cite[Theorem 2.1]{COV1}. Taking $\psi:=u^{q_i}\in L^{\frac{1+q_i}{q_i}}(\mathbb{R}^n,d\sigma_i)$ in \eqref{rieszdualitynormwithgamma1}, we arrive at
\[
\sum_{i=1}^{M}\|\mathcal{I}_{\frac{\alpha}{2}}(u^{q_i} d\sigma_i)\|_{L^2 (\mathbb{R}^n)}\leq C 
\sum_{i=1}^{M}\|u\|^{q_i}_{L^{1+q_i}(\mathbb{R}^n,\;d\sigma_i)}<+\infty,
\]
 which prove \eqref{rieszL2}, and hence $u\in \dot{H}^{\frac{\alpha}{2}}(\R^n)$.
 
Conversely, suppose that $u\in \dot{H}^{\frac{\alpha}{2}}(\R^{n})$ is a positive finite energy solution to \eqref{fractionalproblem}. Then
\[
u=\sum_{i=1}^{M}\mathcal{I}_{\alpha}(u^{q_i}d\sigma_i)+\mathcal{I}_{\alpha}\omega\quad\text{in}\quad\R^n.
\]
By definition of the Riesz potential space $\dot{H}^{\frac{\alpha}{2}}(\R^{n})$ and the semigroup property of Riesz potentials, we have  $(-\Delta)^{\frac{\alpha}{4}}u\in L^{2}(\R^{n})$ and
\[
(-\Delta)^{\frac{\alpha}{4}}u=\sum_{i=1}^{M}\mathcal{I}_{\frac{\alpha}{2}}(u^{q_i}d\sigma_i)+\mathcal{I}_{\frac{\alpha}{2}}\omega\quad a.e.\quad \text{in} \quad \R^n,
\]
for $i=1,2,\dots,M.$ Hence $\mathcal{I}_{\frac{\alpha}{2}}\omega \in L^{2}(\R^n)$ and thus \eqref{rieszpotmu} is valid. The condition \eqref{rieszpotsigma} can be verified by showing that $u\in L^{1+q_i}(\R^n,d\sigma_i)$.
For each nonnegative function $\phi\in L^2(\R^n),$ we have
\[
\int_{\R^n}[(-\Delta)^{\frac{\alpha}{4}}u] \phi \;dx=\sum_{i=1}^{M}\int_{\R^n}[\mathcal{I}_{\frac{\alpha}{2}}(u^{q_i}d\sigma_i)]\phi \; dx+\int_{\R^n}[\mathcal{I}_{\frac{\alpha}{2}}\omega]\phi \;dx
\]
Next, by the Tonelli's Theorem and Schwarz's inequality, we obtain
\be \label{eq0}
\begin{split}
\sum_{i=1}^{M}\Big|\int_{\R^n}u^{q_i }[\mathcal{I}_{\frac{\alpha}{2}}\phi]d\sigma_i\Big|&=\sum_{i=1}^{M} \Big|\int_{\R^n}[\mathcal{I}_{\frac{\alpha}{2}}(u^{q_i}d\sigma_i)]\phi \; dx\Big| \\ &\leq \Big|\int_{\R^n}[(-\Delta)^{\frac{\alpha}{4}}u] \phi \;dx\Big| + \Big|\int_{\R^n} (\mathcal{I}_{\frac{\alpha}{2}}\omega) \phi\;dx\Big| \\
&\leq c\|\phi\|_{L^2{(\R^n)}},
\end{split}\ee
where $c:= \| (-\Delta)^{\frac{\alpha}{4}}u\|_{L^2(\R^n)}+\|\mathcal{I}_{\frac{\alpha}{2}}\omega\|_{L^2(\R^n)}<+\infty,$ since $u\in \dot{H}^{\frac{\alpha}{2}}(\R^n)$ and $ \omega \in \dot{H}^{-\frac{\alpha}{2}}(\R^n)$. Setting $\phi:=  (-\Delta)^{\frac{\alpha}{4}}u,$ which is a nonnegative function of class $L^2(\R^n)$ in \eqref{eq0}, we get
\[
\sum_{i=1}^{M} \|u\|^{1+q_i}_{L^{1+q_i}(\R^n,d\sigma_i)}\leq c \| (-\Delta)^{\frac{\alpha}{4}}u\|_{L^2(\R^n)}<+\infty.
\]
This show that $u\in L^{1+q_i}(\R^n,d\sigma_i)$ for each $i=1,2,\dots,M.$ This, together with Theorem \ref{thm:solofinteq} implies that the first condition  \eqref{rieszpotsigma} is valid.

Now, we are ready to prove the uniqueness result. Suppose that $u$ and $v$ are positive solutions in $\dot{H}^{\frac{\alpha}{2}}(\R^n)$ to \eqref{fractionalproblem}.  We know that  $u=v$ as elements of $\dot{H}^{\frac{\alpha}{2}}(\R^n)$
if $u=v \, d\sigma_i$-a.e. We claim that if $u\geq v\;\text{a.e}$, then $u=v\; d\sigma_i$-a.e.
Suppose that $u\geq v$ a.e.  in $\mathbb{R}^n.$ In fact,  using the testing functions $\phi,\psi \in L^2(\mathbb{R}^n)$, we have
\[
\int_{\mathbb{R}^n} (-\Delta)^{\frac{\alpha}{4}} u \phi \,dx=\sum_{i=1}^{M}\int_{\mathbb{R}^n}  \mathcal{I}_{\frac{\alpha}{2}}( u^{q_i} d\sigma_i)\phi \, dx+ \int_{\mathbb{R}^n}  (\mathcal{I}_{\frac{\alpha}{2}}\omega) \phi \,dx
\]
and
\[
\int_{\mathbb{R}^n} (-\Delta)^{\frac{\alpha}{4}} v \psi \,dx =\sum_{i=1}^{M}\int_{\mathbb{R}^n} \mathcal{I}_{\frac{\alpha}{2}}(v^{q_i} d\sigma_i)  \psi \,dx + \int_{\mathbb{R}^n}  (\mathcal{I}_{\frac{\alpha}{2}}\omega) \psi \,dx. 
\]
Testing the equations
\[
\int_{\mathbb{R}^n} (-\Delta)^{\frac{\alpha}{4}} u \phi \,dx=\sum_{i=1}^{M}\int_{\mathbb{R}^n} u^{q_i} (\mathcal{I}_{\frac{\alpha}{2}} \phi) \,d\sigma_i + \int_{\mathbb{R}^n}  \mathcal{I}_{\frac{\alpha}{2}} \phi \,d\omega,\quad \phi\in L^{2}(\mathbb{R}^n)
\]
and
\be\label{testwithpsi}
\int_{\mathbb{R}^n} (-\Delta)^{\frac{\alpha}{4}} v \psi \,dx =\sum_{i=1}^{M}\int_{\mathbb{R}^n} v^{q_i} (\mathcal{I}_{\frac{\alpha}{2}} \psi) \,d\sigma_i + \int_{\mathbb{R}^n}  \mathcal{I}_{\frac{\alpha}{2}} \psi \,d\omega,\quad \psi\in L^{2}(\mathbb{R}^n)
\ee
with  $\phi=(-\Delta)^{\frac{\alpha}{4}} u$ and $\psi=(-\Delta)^{\frac{\alpha}{4}}v$,  respectively,  we obtain
\[
\|u\|_{\dot{H}^{ \frac{\alpha}{2}}(\mathbb{R}^n)}^2 = \int_{\mathbb{R}^n} \big((-\Delta)^{\frac{\alpha}{4}}u \big)^2 dx =\sum_{i=1}^{M}\int_{\mathbb{R}^n} u^{1+q_i}\; d\sigma_i +\int_{\mathbb{R}^n} u \;d\omega
\]
and
\[
\|v\|_{\dot{H}^{\frac{\alpha}{2}}(\mathbb{R}^n)}^2 = \int_{\mathbb{R}^n} \big((-\Delta)^{\frac{\alpha}{4}}v \big)^2 dx =\sum_{i=1}^{M}\int_{\mathbb{R}^n} v^{1+q_i}\; d\sigma_i +\int_{\mathbb{R}^n} v \;d\omega.
\]
By applying the  Discrete Hidden convexity,  [see \cite{BF},  Proposition 4.2], along the curve
\[
\Gamma_t (x):=[(1-t)v^2(x) +tu^2(x)]^\frac{1}{2},\quad t \in [0,1],
\]
we obtain
\[
|\Gamma_t (x) -\Gamma_t (y)|^2 \leq (1-t) |v(x)-v(y)|^2 + t|u(x)-u(y)|^2
\]
\begin{align*}
\int_{\mathbb{R}^n}\int_{\mathbb{R}^n} \frac{|\Gamma_t (x) - \Gamma_t (y)|^2}{|x-y|^{n+\alpha}}\;dx\;dy&\leq (1-t)\int_{\mathbb{R}^n}\int_{\mathbb{R}^n} \frac{| v(x) - v(y)|^2}{|x-y|^{n+\alpha}}\;dx\;dy\\
&\quad\;+t\int_{\mathbb{R}^n}\int_{\mathbb{R}^n} \frac{|u(x) - u(y)|^2}{|x-y|^{n+\alpha}}\;dx\;dy
\end{align*}
Since $\|.\|_{\dot{H}^{\frac{\alpha}{2}}(\R^n)}$ is equivalent to the Gagliardo seminorm, 
\[
\int_{\mathbb{R}^n} \big((-\Delta)^{\frac{\alpha}{4}}\Gamma_t \big)^2 dx \leq (1-t) \int_{\mathbb{R}^n} \big((-\Delta)^{\frac{\alpha}{4}} v \big)^2 dx +t \int_{\mathbb{R}^n} \big((-\Delta)^{\frac{\alpha}{4}} u \big)^2 dx
\]
\begin{align*}
\int_{\mathbb{R}^n} \frac{ \big[ \big((-\Delta)^{\frac{\alpha}{4}}\Gamma_t  \big)^2 - \big((-\Delta)^{\frac{\alpha}{4}}\Gamma_0 \big)^2 \big] }{t} \,dx &\leq \int_{\mathbb{R}^n} \big[   \big((-\Delta)^{\frac{\alpha}{4}} u \big)^2   - \big((-\Delta)^{\frac{\alpha}{4}} v \big)^2  \big] \,dx\\
 &\leq \sum_{i=1}^{M}\int_{\mathbb{R}^n}     (u^{1+q_i}-v^{1+q_i})\,d\sigma_i  +\int_{\mathbb{R}^n} (u-v)\,d\omega.
\end{align*}
We deduct the above equation by applying the following inequality
\[
a^2-b^2 \geq 2 b\cdot(a-b)\quad \text{for}\;\; a,b\in \mathbb{R}
\] to the left-hand side of the equation, and hence we obtain 
\be \label{equationgamma}
2\int_{\mathbb{R}^n}  (-\Delta)^{\frac{\alpha}{4}}\Gamma_0  \frac{\big((-\Delta)^{\frac{\alpha}{4}}\Gamma_t  - (-\Delta)^{\frac{\alpha}{4}}\Gamma_0 \big) }{t}\,dx \leq \sum_{i=1}^{M}\int_{\mathbb{R}^n}     (u^{1+q_i}-v^{1+q_i})\,d\sigma_i  +\int_{\mathbb{R}^n} (u-v)\,d\omega.
\ee
Testing \eqref{testwithpsi} by $\psi=(-\Delta)^{\frac{\alpha}{4}}(\Gamma_t-\Gamma_0)\in L^2(\mathbb{R}^n),$ we find that
\be \label{testwithgamma}
2\int_{\mathbb{R}^n}(-\Delta)^{\frac{\alpha}{4}}v \frac{\big((-\Delta)^{\frac{\alpha}{4}}\Gamma_t  - (-\Delta)^{\frac{\alpha}{4}}\Gamma_0 \big) }{t} \,dx= 2 \sum_{i=1}^{M} \int_{\mathbb{R}^n}\frac{v^{q_i}(\Gamma_t - \Gamma_0)}{t}\,d\sigma_i +2\int_{\mathbb{R}^n} \frac{\Gamma_t - \Gamma_0 }{t}\,d\omega.
\ee
Thus, by \eqref{equationgamma} and \eqref{testwithgamma}, we have
\be \label{aye}
2\sum_{i=1}^{M}\int_{\mathbb{R}^n}v^{q_i} \frac{(\Gamma_t - \Gamma_0)}{t}\,d\sigma_i +2\int_{\mathbb{R}^n}\frac{\Gamma_t - \Gamma_0}{t}\,d\omega \leq \sum_{i=1}^{M}\int_{\mathbb{R}^n} (u^{1+q_i}-v^{1+q_i})\,d\sigma_i  +\int_{\mathbb{R}^n} (u-v)\,d\omega.
\ee
Meanwhile, by Fatou's Lemma,
\be \label{sigmalimit}
\int_{\mathbb{R}^n} v^{q_i} ( \frac{u^2 - v^2}{v} )\,d\sigma_i \leq \liminf_{t\to 0} 2\int_{
\mathbb{R}^n} v^{q_i} (\frac{\Gamma_t -\Gamma_0}{t})\,d\sigma_i,
\ee
and
\be \label{mulimit}
\int_{\mathbb{R}^n}  \frac{u^2 - v^2}{v}\,d\omega \leq \liminf_{t\to 0} 2\int_{
\mathbb{R}^n}  (\frac{\Gamma_t -\Gamma_0}{t})\,d\omega
\ee
for each $i.$
Since \eqref{aye} holds for all $t\in[0,1]$, it follows from \eqref{sigmalimit} and \eqref{mulimit} that
\[
\sum_{i=1}^{M}\int_{\mathbb{R}^n} \frac{v^{q_i} (u^2 -v^2)}{v}\,d\sigma_i +\int_{\mathbb{R}^n} \frac{u^2-v^2}{v}\,d\omega \leq \sum_{i=1}^{M}\int_{\mathbb{R}^n} (u^{1+q_i}-v^{1+q_i})\,d\sigma_i  +\int_{\mathbb{R}^n} (u-v)\,d\omega.
\]
that is,
\[
\sum_{i=1}^{M}\int_{\mathbb{R}^n} (\frac{v^{q_i} u^2}{v}- u^{1+q_i})\,d\sigma_i +\int_{\mathbb{R}^n} (\frac{u^2}{v}-u)\,d\omega \leq 0.
\]
Here both integrals in the left-hand side are nonnegative since $u\geq v\; d\sigma_i$-a.e.  and $u\geq v \;d\omega$-a.e. Indeed,
\[
\sum_{i=1}^{M}\int_{\mathbb{R}^n} (\frac{v^{q_i} u^2}{v}- u^{1+q_i})\,d\sigma_i =\sum_{i=1}^{M}\int_{\mathbb{R}^n} ( \frac{u^{1+{q_i}} v^{q_i} ( u^{1-q_i} - v^{1-q_i})}{v})\, d\sigma_i \geq 0
\]
and
\[
\int_{\mathbb{R}^n} (\frac{u^2}{v}-u)\,d\omega=\int_{\mathbb{R}^n} \frac{u^2 -uv}{v}\,d\omega\geq 0.
\]
Therefore,  both integrals must vanish,  and thus $u=v \;d\sigma_i$-a.e. and $u=v\; d\omega$-a.e. 
By definition of solution \eqref{solofinteq}, $u=v$ in $\mathbb{R}^n$.
\end{proof}
\begin{Rem}
By applying the convexity of the Dirichlet integrals $\int_\Omega |\nabla\cdot|^2 dx$, we may use in the same way as in the proof of Theorem \eqref{thm:fractional} in the case $\alpha=2$ to obtain the uniqueness of a positive finite energy solution in $\dot{W}^{1,2}_0(\Omega)$ to \eqref{classical}. 
\end{Rem}

\subsection*{Conflict of interest} 
All authors declare no conflicts of interest in this paper. 

\subsection*{Acknowledgement}
This work is supported by the Thammasat University Research Unit in Gait Analysis and Intelligent Technology (GaitTech).  A.C.M.  gratefully acknowledges financial support from the Excellent Foreign Student (EFS) scholarship,  Sirindhorn International Institute of Technology (SIIT),  Thammasat University.
The authors would like to express our gratitude to Professor Igor E. Verbitsky for providing us useful comments regarding uniqueness results in \cite{PV2} and \cite{V4}.
\bibliographystyle{abbrv} 
\bibliography{reference_MS1}
\end{document}